\documentclass[a4paper,11pt]{article}
          
\usepackage[paperwidth=7in, paperheight=10in, margin=.875in]{geometry}
 \usepackage[backref,colorlinks,linkcolor=red,anchorcolor=green,citecolor=blue]{hyperref}
\usepackage{amsfonts,amssymb,amsthm}
\usepackage{amsmath}
\usepackage{graphicx}
\usepackage{cite}
\usepackage{enumerate}
  \newtheorem{theorem}{Theorem}
  \newtheorem{lemma}[theorem]{Lemma}

  \newtheorem{definition}[theorem]{Definition}
  \newtheorem{remark}[theorem]{Remark}

   \allowdisplaybreaks
\begin{document}
 \title{Improved duality estimates: time discrete case and applications to a class of cross-diffusion systems.}

          \author{Thomas Lepoutre\thanks{Univ Lyon, Inria, Université Claude Bernard Lyon 1, CNRS UMR5208, Institut Camille Jordan, F-69603 Villeurbanne, France, (thomas.lepoutre@inria.fr). http://math.univ-lyon1.fr/homes-www/lepoutre/index.html}
          }

         \maketitle

          \begin{abstract}
             We adapt the improved duality estimates for bounded coefficients derived by Canizo et al. to the framework of cross diffusion. Since the estimates can not be directly applied we need to derive a time discrete version of their results and apply it  to an implicit semi-discretization in time of the cross diffusion systems. This leads to new global existence results for cross diffusion systems with bounded cross diffusion pressures and potentially superquadratic reaction. 
          \end{abstract}

\section{Introduction.}\label{sec:intro}

The following manuscript is devoted to the adaptation of improved duality estimates introduced in \cite{canizo2014improved} to a time discrete setting. As an application we extend some recent existence results in cross-diffusion-reaction models in Laplace form. 
Cross-diffusion appears in ecology, chemistry or semiconductor modelling. The systems we have in mind have been introduced by Shigesada Kawasaki and Teramoto in \cite{Shigesada1979} and have given birth to a large literature. The original system has the form ($\Omega$ is a smooth bounded domain)
\begin{align}\label{eq:SKT}
\begin{cases}
\partial_t u-\Delta \left(d_1+a_{11}u+a_{12}v\right)u=u(R_1-r_{11}u-r_{12}v),\\
\partial_t v-\Delta \left(d_2+a_{21}u+a_{22}v\right)v=v(R_2-r_{21}u-r_{22}v),\\
\partial_nu=\partial_nv =0,\quad \text{ on }\partial\Omega.
\end{cases}
\end{align}
For the most part, the local existence and uniqueness of strong solutions in a very general settings stem from the seminal work of Herbert Amann \cite{amann90b}. Passing from local to global existence for \eqref{eq:SKT} remains an open problems except if strong additional structural assumptions are added on the coefficients $a_{ij}$ see \cite{Yagi,2015_article_HoangNP_SJMA}. We are here interested in the question of weak solutions for such systems. Such solutions have been studied by Jungel and coauthors for \eqref{eq:SKT} \cite{Galiano_Num_Math,Chen2004,Chen2006} through the fundamental remark that the original system possesses an entropic structure. Such solutions are global in time (but uniqueness is lost except in very specific situations \cite{chen2017note}).  This entropic structure has been considerably exploited and generalized in several complementary directions. The first one has been introduced by Burger and coauthors in \cite{burger} and generalized in \cite{jungbound} consists in considering systems in which the gradient dissipation implied in the entropy dissipation leads to boundedness. This is the so called boundedness by entropy principle. If boundedness can not be obtained, one might need additional estimates besides the entropy dissipation. To solve this, a direction that has been considered since \cite{Lepoutre_JMPA} is based on the parabolic structure of  cross-diffusion models with a Laplace structure, namely of the form
\begin{align}
\label{eq:cross_diff_laplace}
\begin{cases}
\partial_tu_i-\Delta p_i(U)u_i=u_ir_i^-(U)+r_i^+(U)=R_i(U), \quad \text{in }\Omega, \, 1\leq i\leq I,\\
\partial_n(p_i(U)u_i)=0.
\end{cases}
\end{align}
For such systems, one can have additional estimates derived in the context of reaction-diffusion systems in \cite{PiSc}, applied to cross-diffusion models in \cite{PierreD} for a specific case and in a more systematic way in \cite{Lepoutre_JMPA,DesLepMou,Lepoutre_Desvillettes_Moussa_Trescases,lepoutre2017entropic}. This can be resumed in the fact that under general hypothesis we can deduce an a priori bounds on $\|U\|_{L^2(Q_T)}$ ($Q_T$ standing for $\Omega\times (0,T)$). This comes from a very elegant duality argument or direct time integration of the equation. 

In the context of chemical reaction diffusion systems, these duality estimates have been considerably precised by Canizo et al. in \cite{canizo2014improved} in the specific case of bounded diffusion coefficients. Their results insures  $L^p(Q_T)$ estimates where $p>2$ depends on the domain, and the lower and upper bound of the coefficients. It has been applied to several situations concerning systems with constant but species-dependent diffusion coefficients \cite{canizo2014improved,fellner2016global} with an infinite number of species \cite{breden2017smoothness} or in the cross-diffusion triangular setting \cite{DesTres}.

\textbf{Approximation difficulty}. It has been noticed in former works \cite{DesLepMou,Lepoutre_Desvillettes_Moussa_Trescases,lepoutre2017entropic} that the entropic structure and the duality estimate are of different nature and that it is therefore a nontrivial problem to build solutions respecting both structures. In case of reaction diffusion system, a truncation of  nonlinearities allows to apply estimate to smooth approximation that are robustly preserving the estimates. For cross diffusion systems, this is known to be more difficult. One of the most commonly used way of building solutions is a time discretization. The approximation scheme based on entropic variables and exploited in \cite{Chen2006} and \cite{burger} is generally not suitable if one wants to keep the duality estimate at the limit. An implicit in time discretization scheme (also called Rothe method) has been developed in  \cite{Lepoutre_Desvillettes_Moussa_Trescases} and generalized in \cite{lepoutre2017entropic}, and it has proved to be an approach combining the power of both structures. In the cases $p=2$ there is no additional cost (the constant involved in the estimates are the same). For $L^p, p>2$ this is not the case. We do not keep the optimal constants throughout the process and therefore cannot ensure that we can reach the optimal $L^p$ integrability. On the other hand we are still in position of ensuring better than $L^2$ integrability in the cases of bounded $p_i$. In this manuscript, we show that a discrete version of the estimates derived in \cite{canizo2014improved} also applies to this approximation procedure, allowing extension of existence results to a larger class of reaction terms in case of bounded diffusion (from above and below) $p_i$. 

The paper is organized as follows:
In section~\ref{sec:main}, we remind the structural hypothesis we make on system \eqref{eq:cross_diff_laplace}, remind the duality estimates from \cite{canizo2014improved}. We then give the time discrete equivalent and state our main results on \eqref{eq:cross_diff_laplace} in the case of bounded $p_i$. In section~\ref{sec:duality}, we establish the proof of time-discrete estimates and their consequences for semi-discretization of parabolic equations. In section~\ref{sec:cross_appli}, we illustrate our results through simple examples.

\section{Framework and statement of the main results.}\label{sec:main}

\subsection{Preliminary hypothesis on \eqref{eq:cross_diff_laplace}.}

\textbf{Structural hypothesis on \eqref{eq:cross_diff_laplace}.}
Concerning systems of the form \eqref{eq:cross_diff_laplace}, 
A vectorial notation is then the following $U=(u_i)_{1\leq i\leq I}, A(U)=(p_i(U)u_i)_{1\leq i\leq I}$.
\begin{align}\label{eq:cross_diff_laplace_vect}
\partial_t U-\Delta A(U)=R(U).
\end{align}
And the divergence form is the following 
\begin{align}\label{eq:cross_diff_laplace_div}
\partial_t U-div\left(DA(U)\nabla U\right)=R(U).
\end{align}
We will make the following structural hypothesis (see \cite{Lepoutre_Desvillettes_Moussa_Trescases} and \cite{lepoutre2017entropic} for more details).
\paragraph{Regularity assumption on the coefficients}\

In what follows we will make the following hypothesis 
\begin{align}\label{eq:p_i_reg}
p_i\in C^0(\mathbb{R}_+^I,\mathbb{R}_+)\cap C^1((\mathbb{R}_+^*)^I,\mathbb{R}_+),
\end{align}
\begin{align}\label{eq:r_i_reg}
r_i^\pm\in C^0(\mathbb{R}_+^I,\mathbb{R}_+)\cap C^1((\mathbb{R}_+^*)^I,\mathbb{R}_+),
\end{align}
We will also make the following assumption on $A$: 
\begin{align}\label{eq:Ainversible}
A:\mathbb{R}_+^I \mapsto \mathbb{R}_+^I\text{ is a homeomorphism.}
\end{align}

\paragraph{Entropy dissipation control}\
\begin{definition}\label{def:entropy}
We say that the system \eqref{eq:cross_diff_laplace} admits a nondegenerate entropy if there exists a convex $ C^2 $functional $\mathcal{H}:(\mathbb{R}_+^*)^I\mapsto \mathbb{R}_+$ such that 
\begin{align}
D^2\mathcal{H}(U)DA(U)>0. \label{eq:entropy_dissip}
\end{align} 
It is said to be compatible with $R$ if we have additionally 
\begin{align}
\nabla\mathcal{H}(U).R(U)\leq C_\mathcal{H}\left(1+\sum U_i+\mathcal{H}(U)\right)\label{eq:entropy_compatible}
\end{align}
The entropy is said to be uniform is there exists positive continuous function $f_i$ such that if we denote $\mathcal{D}_f$ the diagonal matrix with $(\mathcal{D}_f)_{ij}=f_i(u_i)\delta_{i=j}$, we have
\begin{align}
D^2\mathcal{H}(U)DA(U)\geq \mathcal{D}_f. \label{eq:entropy_uniform}
\end{align} 
In the sense of symmetric matrices.
\end{definition}
\begin{definition}
The reaction terms are called mass controlling if there exists a positive constant vector $\Phi>$ and a constant $C_R\geq 0$ such that 
\begin{align}\label{eq:masscontrol}
\forall U\geq 0,\quad \sum_i R_i(U)\leq C_R(1+\sum u_i).
\end{align}
\end{definition}
Note that the hypothesis \eqref{eq:masscontrol} immediately imply the following estimates
\begin{align}
\int_\Omega u_i(t)\leq Ke^{C_Rt}, \label{eq:mass_control_estimate}\\
\int_\Omega \mathcal{H}(U)(t) +\int_0^t e^{C(t-s)}\int_\Omega \nabla U D^2\mathcal{H}(U)DA(U)\nabla U\leq e^{Ct}\int_\Omega \mathcal{H}(U^0)+K' e^{(C+C_R)t}\label{eq:entropy_control_estimate}
\end{align}
The hypothesis on the $p_i$ together with \eqref{eq:masscontrol} ensure the following time and space estimate which is at the heart of our construction 
\begin{align}\label{eq:dualite}
\int_0^T\int_\Omega (\sum u_i\sum p_i u_i)&+\int_\Omega\left|\nabla\int_0^T \sum p_i u_i e^{-C_R (t-s)}\right|^2 \nonumber\\
&\leq C(R,A,\Omega,T)\left(\|U^0\|_1+\left\|\sum u_i^0-\langle \sum u_i^0\rangle\right\|_{H^{-1}(\Omega)}\right).
\end{align}
While computation leading to \eqref{eq:mass_control_estimate},\eqref{eq:entropy_control_estimate} are quite standard, we remind shortly in the appendix the arguments leading to \eqref{eq:dualite}.

We remind here the existence results that is allowed by such a structure. 
\begin{theorem}[\cite{lepoutre2017entropic}]\label{thm:mainLepMou}
  Let $\Omega$ be a smooth domain. Assume \eqref{eq:p_i_reg},\eqref{eq:r_i_reg},\eqref{eq:Ainversible} and that there exists uniform compatible entropy (satisfying \eqref{eq:entropy_dissip},\eqref{eq:entropy_compatible},\eqref{eq:entropy_uniform}). 
Assume finally that the function $R$ satisfies (for some norm $\|\cdot\|$ on $\mathbb{R}^I$)
\begin{align}\label{eq:equiint}
\| R(X)\|=\textnormal{o}\left(\left(\sum_{i=1}^I p_i(X)x_i\right)\left(\sum_{i=1}^I x_i\right)+\mathcal{H}(X)\right), \text{ as } \|X\| \rightarrow \infty.
\end{align}
Then, for any $0\leq U_{\textnormal{in}} \in L^1(\Omega)\cap  H^{-1}(\Omega)$,
 such that $ \mathcal{H}(U_{\textnormal{in}})\in L^1(\Omega)$, 
 there exists $0\leq U \in L^1(Q_T)$ such that $A(U)\in L^1(Q_T)$ and $R(U)\in L^1(Q_T)$ which is a weak solution of system \eqref{eq:cross_diff_laplace} with initial data $U_{\textnormal{in}}$ and homogeneous Neuman boundary conditions, \emph{i.e.} for  all $\Psi\in C^1_c([0,T); C^2(\overline{\Omega})^I)$ satisfying $\partial_n \Psi = 0$ on $\partial \Omega$, there holds 
 \begin{align}\label{eq:veryweak}
- \int_\Omega U_{\textnormal{in}} \cdot\Psi(0,\cdot) = \int_{Q_T}\Big( U.\partial_t\Psi+A(U) \cdot \Delta \Psi + R(U) \cdot \Psi\Big).
\end{align}
Moreover, this solution satisfies the following estimate on $[0,T]$:
\begin{align}
\int_\Omega\mathcal{H}(U(t)) +\int_0^t\int_{\Omega} \langle \nabla U, D^2(\mathcal{H})(U) D(A)(U)\nabla U\rangle  &\leq (1+e^{CT})\left(1+\int_\Omega \mathcal{H}(U_{\textnormal{in}})\right),
\label{eq:gradient_control}
\end{align}
where $C$ is a combination of $C_\mathcal{H}$ and $C_R$. 
\end{theorem}
\begin{remark}
For sake of simplicity we have chosen a mass control but everything done here works if we replace \eqref{eq:masscontrol} by the existence of a positive vector $\phi>0$ such that $\phi.R(U)\leq C_R(1+\phi.U)$. Essentially, all the sums in \eqref{eq:dualite} have to be replaced by weighted sums $(\sum v_i\rightarrow \sum \phi_i v_i)$. 
\end{remark}
This results is very large and its main constraint is in practice the control of reaction. In most situation of interest, the equation \eqref{eq:equiint} does not allow to treat standard logistic reaction terms. For the system \eqref{eq:SKT} it is not a real problem, because additional equiintegrability is directly given from the entropy dissipation inequality \cite{Chen2006} offering a gain of a priori $L^2log L$ integrability. In \cite{Lepoutre_Desvillettes_Moussa_Trescases}, we have treated  a general case for power like $p_i$. Let us consider the system,
\begin{align}\label{eq:SKTconcave}
\begin{cases}
\partial_t u-\Delta \left(d_1+v^\alpha\right)u=u(1-u-v),\\
\partial_t v-\Delta \left(d_2+u^\beta\right)v=0,\\
\partial_nu=\partial_nv =0,\quad \text{ on }\partial\Omega.
\end{cases}
\end{align}
As soon as we have $\alpha\beta\leq 1$, the system possesses an entropy structure satisfying \eqref{eq:entropy_compatible} given by 
$$
\mathcal{H}(U)= \frac{u^\beta-\beta u+\beta-1}{\beta(\beta-1)}+ \frac{v^\alpha-\alpha u+\alpha-1}{\alpha(\alpha-1)},
$$
Surprisingly, when $\beta$ is large, the control given by the entropy allows to obtain the necessary control \eqref{eq:equiint} to treat the quadratic terms $u^2$ in the reaction. When $\beta$ is small, we can not establish existence through theorem~\ref{thm:mainLepMou}. The results from \cite{lepoutre2017entropic} however can cover quadratic reaction terms when the entropy gives an important control or in presence of self-diffusion (as a verification that $|X|^2$ satisfy then \eqref{eq:equiint}). In low dimension the gradient control coming out of \eqref{eq:entropy_control_estimate} can give enough additional integrability through Sobolev emebeddings. What follows gives a solution for the case where one replaces $u^\beta$ by a bounded function (keeping all the other hypotheses true).

\subsection{Improved duality estimates: continuous and discrete case.}
In \cite{canizo2014improved}, a breakthrough was obtained and applied to reaction diffusion systems. Firslty one needs to introduce an important notation 
\begin{definition}\label{def:Cmp}
Let $\Omega$ be a smooth domain of $\mathbb{R}^N$, for all $p\in]1,\infty[$ and $m>0$ there exists a constant denoted $C_{m,p}$ independent of $T$ such that the solution to 
$$
\begin{cases}
\partial_t w-m\Delta w=f,\\
\partial_n w=0,\\
w(0,x)=0.
\end{cases}
$$
satisifies 
$$
\|\Delta w\|_{L^p(Q_T)}\leq C_{m,p} \|f\|_{L^p(Q_T)}.  
$$
\end{definition}
In \cite{canizo2014improved} it is applied to equalities, so we prefer to refer to an inequality version in the form of a stability principle. 
\begin{lemma}[adapted from Proposition 2.5 in \cite{breden2017smoothness}]\label{lem:dualiteamelioree}
Let $M\geq 0$ be a lower and upper bounded function:  $0<a\leq M(t,x)\leq b<+\infty$. Let $p\in]1,\infty[$ satisfy  
\begin{align}\label{eq:condCab}
\frac{b-a}{2}C_{\frac{a+b}{2},p}<1.
\end{align}
Assume $u\geq 0$ satisifies weakly
$$ 
\begin{cases}
\partial_tu-\Delta Mu \leq C(1+u),\\
\partial_n (Mu)=0,\\
u^0\in L^{p}(\Omega),
\end{cases} 
$$
then $u\in L^p (Q_T)$ and we have the following a priori estimate 
$$
\|u\|_{L^p(Q_T)}\leq C(1+\|u^0\|_{L^p(\Omega)}),
$$
where $C$ depends only on $\Omega,a,b,T$.
\end{lemma} 

This results gives room to improvement of our results in the case of bounded coefficients. In principle, in case where \eqref{eq:masscontrol} is satisfied and the $p_i$ satisfy $0<a \leq p_i\leq b<+\infty$ and $p\in]2,\infty[$ satisfying 
\eqref{eq:condCab}, we shall be able to extend theorem~\ref{thm:mainLepMou} replacing condition \eqref{eq:equiint} by  $$U^0\in L^p(\Omega)^I, \|R(X)\|=o(1+|X|^p).$$

However, as it has been noticed in \cite{DesLepMou,Lepoutre_Desvillettes_Moussa_Trescases,lepoutre2017entropic}, the construction of solution to \eqref{eq:cross_diff_laplace} is not immediate. In particular, combining duality estimates and entropy dissipation is quite difficult. An approximation procedure for \eqref{eq:cross_diff_laplace} has been developed in \cite{Lepoutre_Desvillettes_Moussa_Trescases,lepoutre2017entropic}  that preserves both entropy dissipation and $L^2$ duality estimates.  This construction is based on a time implicit discretization, solving a Euler backwards version of \eqref{eq:cross_diff_laplace}. Time discrete equivalent of \eqref{eq:entropy_control_estimate} and (surprisingly)\eqref{eq:dualite} can be then derived. The adaptation of lemma~\ref{lem:dualiteamelioree} is in fact much more demanding. We shall see in the sequel that there is a discrete equivalent of  definition~\ref{def:Cmp} but there is no guarantee (apart for the fundamental case $p=2$) that the discrete equivalent of $C_{m,p}$ has the same value. This is why we have to restrict our result following the discrete version of the Meyers estimate whose proof can be found in \cite{2002_article_AshyrPW_NFAO} or \cite{2016_article_KovacLL_SJNA}.
\begin{lemma}[Ashyrakyev,Piskarev and Weis, Remark 5.2]

Let us denote $\Psi=\Psi(\tau, m,F)$  the solution of 
$$
\begin{cases}
\frac{\Psi^{k+1}-\Psi^{k}}{\tau} +m\Delta\Psi^{k}=F^{k+1},\\
\Psi^{N}=0,\; \partial_n\Psi^k=0.
\end{cases}
$$
then there exists a constant $K_{p,m}$ that depends only on $\Omega,p,m$ such that for any $F$ in $l^p(L^p)$, we have  $$\left(\sum_{k=0}^{N-1}\tau\|\Delta\Psi^k\|_{L^p( \Omega)}^p\right)^{1/p}\leq K_{m,p} \left(\sum_{k=0}^{N-1}\tau\| F^k\|_{L^p( \Omega)}^p\right)^{1/p}.$$
\end{lemma}
As for the continuous case \cite{1987_article_Lambe_JoFA}, it is remarkable that the constant does not depend on the time horizon ($T$ or $N$).
\subsection{Application to cross diffusion systems.}
We are now in position to state our main theorem 
\begin{theorem}\label{th:main}
Let the hypothesis of theorem~\ref{thm:mainLepMou} hold. Assume additionally that the $p_i$ are bounded from above and below 
\begin{align}
0<a\leq p_i(U)\leq b<+\infty.\label{eq:p_iborne}
\end{align}
Assume $p\in]2,\infty[$ is such that,
\begin{align}
\frac{b-a}{2}K_{a,b,p}<1\label{eq:condKab}.
\end{align}
Then  the conclusion of the theorem~\ref{thm:mainLepMou} holds true adding the hypothesis $U^0\in L^p(\Omega)^I$ and replacing \eqref{eq:equiint} by 
\begin{align}\label{eq:hypRp}
|R(U)|=o(|U|^p).
\end{align}
\end{theorem}
\begin{remark}
As it is the case for \eqref{eq:condCab}, we shall see that there always exists $p$ such that \eqref{eq:condKab} holds true since it is always valid for $p=2$. 
\end{remark}

\section{Improved duality estimates: discrete  case\label{sec:duality}.}

\subsection{The estimates on dual problem.}

As for the time continuous case, the crucial point is that the constant does not depend on the horizon (represented here by $N$ instead of $T$) nor the time step $\tau$. The result can in fact be generalized to $L^p(0,T;L^q(\Omega)$ spaces. Note that the dependency on $m$ takes the form $K_{m,p}\leq K_{1,p}/m$. 
 
Following the lines of \cite{canizo2014improved}, we give an estimate of the value of $K_{m,2}$. 
\begin{lemma}
The constant $K_{m,p}$ satisfies 
$$
K_{m,p}=\frac{K_{1,p}}{m}.
$$
The case $p=2$ is given by 
$$
K_{m,2}=\frac{1}{m}
$$
\end{lemma}
\begin{proof}
The proof is simply based on a the equality (putting both members to the square)
\begin{align*}
\int_\Omega \left(\frac{\Psi^{k+1}-\Psi^k}{\tau}\right)^2+m^2(\Delta\Psi^k)^2= \int_\Omega (F^{k+1})^2-m\int_\Omega \frac{\Psi^{k+1}-\Psi^k}{\tau}\Delta\Psi^{k}\\
\leq  \int_\Omega (F^{k+1})^2+\frac{m}{2\tau}\int_\Omega (\nabla\Psi^{k+1})^2-(\nabla\Psi^{k})^2
\end{align*}
Summation over $k$ gives the result (we remind that $\psi^N=0$). Note that even if we are only interested by the inequality, this is in fact an equality (we just need to consider $N=1$ and $F^1$ is an eigenvector of the laplacian associated to a large eigenvalue to approach equality case). \end{proof}

The second important point is just a consequence of interpolation between $L^p$ spaces. 
\begin{lemma}
Let $p'$ be the conjugate exponent of $p$ such that $1/p+1/p'=1$, then we have $K_{m,p}=K_{m,p'}$. Furthermore if we have $p<r<q$ and $0<\theta<1$ such that 
$$
\frac{1}{r}=\frac{\theta}{p}+\frac{1-\theta}{q},
$$
then we have 
$$
K_{m,r}\leq K_{m,p}^\theta K_{m,q}^{1-\theta}.
$$
\end{lemma}
This leads to the main consequence for adaptation of results of \cite{canizo2014improved}.
\begin{lemma}
Consider solutions of the problem 
\begin{align}\label{eq:dual_discret}
\begin{cases}
\frac{\Psi^{k+1}-\Psi^{k}}{\tau} +a^{k+1}\Delta\Psi^{k}=F^{k+1},\\
\Psi^{N}=0,\; \partial_n\Psi^k=0.
\end{cases}
\end{align}
with smooth $a^{k+1}$ satisfying $0<a<a^{k+1}<b<+\infty$, assume that $F\in l^p(L^p)$ with  $p$ satisfying 
$$
\frac{b-a}{2}K_{\frac{a+b}{2},p}<1,
$$
then we have the following estimates:
\begin{align}\label{eq:imp_laplace}
\left(\sum_{k=0}^{N-1} \tau \|\Delta\psi^k\|_{L^p(\Omega)}^p\right)^{1/p}\leq \bar D_{a,b,p}\left(\sum_0^{N-1} \tau \|F^{k+1}\|_{L^p(\Omega)}^p\right)^{1/p}
\end{align}
\begin{align}\label{eq:imp_dt}
\|\psi^0\|_{L^p(\Omega)}\leq \left(1+b\bar D_{a,b,p}\right)(N\tau)^{1/p'}\left(\sum_0^{N-1} \tau \|F^{k+1}\|_{L^p(\Omega)}^p\right)^{1/p}
\end{align}
where 
$$
\bar D_{a,b,p}=\frac{K_{\frac{a+b}{2},p}}{1-\frac{b-a}{2}K_{\frac{a+b}{2},p}}
$$
\end{lemma}
\begin{proof}
The proof  follows the lines of the proof of lemma 2.2 in  \cite{canizo2014improved}.  We simply write 
$$
\frac{\Psi^{k+1}-\Psi^{k}}{\tau} +\frac{a+b}{2}\Delta\Psi^{k}=F^{k+1}+\left(\frac{a+b}{2}-a^{k+1}\right)\Delta\Psi^{k}
$$
Using the previous lemma, we have immediately
\begin{align*}
\left(\sum_{k=0}^{N-1} \tau \|\Delta\psi^k\|_{L^p(\Omega)}\right)^{1/p}&\leq K_{\frac{a+b}{2},p}\left(\sum_0^{N-1} \tau \left\|F^{k+1}+\left(\frac{a+b}{2}-a^{k+1}\right)\Delta\Psi^{k}\right\|_{L^p(\Omega)}^p\right)^{1/p}
\\
&\leq K_{\frac{a+b}{2},p}\left(\sum_0^{N-1} \tau \|F^{k+1}\|_{L^p(\Omega)}^p\right)^{1/p}\\
&+K_{\frac{a+b}{2},p}\left(\sum_0^{N-1} \tau \left\|\frac{a+b}{2}-a^{k+1}\right\|_\infty\|\Delta\Psi^{k}\|_{L^p(\Omega)}^p\right)^{1/p}
\end{align*}
Since by construction, we have $$\left\|\frac{a+b}{2}-a^{k+1}\right\|_\infty\leq \frac{b-a}{2},$$
we end up with 
$$
\left(\sum_{k=0}^{N-1} \tau \|\Delta\psi^k\|_{L^p(\Omega)}\right)^{1/p}\left(1-\frac{b-a}{2}K_{\frac{a+b}{2},p}\right)\leq K_{\frac{a+b}{2},p}\left(\sum_0^{N-1} \tau \|F^{k+1}\|_{L^p(\Omega)}^p\right)^{1/p}.
$$
Leading immediately to equation\eqref{eq:imp_laplace}. To obtain \eqref{eq:imp_dt}, we remark 
\begin{align*}
\Psi^0=\sum_{k=0}^{N-1}\tau \left(F^{k+1}-a^{k+1}\Delta\Psi^{k}\right)
\end{align*}
Therefore, we have immediately 
\begin{align*}
\|\Psi^0\|_p&\leq\sum_{k=0}^{N-1}\tau \left(\|F^{k+1}\|_p+\|a^{k+1}\Delta\Psi^{k}\|_p\right)\\
&\leq \sum_{k=0}^{N-1}\tau \left(\|F^{k+1}\|_p+b\|\Delta\Psi^{k}\|_p\right)\\
&\leq (N\tau)^{1/p'}\left(\left(\sum_0^{N-1} \tau \|F^{k+1}\|_{L^p(\Omega)}^p\right)^{1/p}+b\left(\sum_{k=0}^{N-1} \tau \|\Delta\psi^k\|_{L^p(\Omega)}^p\right)^{1/p}\right).
\end{align*}
Applying \eqref{eq:imp_laplace} we obtain the result \eqref{eq:imp_dt} and thereby the lemma.

\end{proof}
\begin{remark}\label{rem:notequal}
Apart from the case $p=2$, we cannot insure the equality between $K_{m,p}$ and $C_{m,p}$ in general. Note that we have also by this mean a general estimate on $\|\Psi^k\|_p$
$$
\|\Psi^k\|_p\leq\left(1+\bar D_{a,b,p}\right)((N-k)\tau|\Omega|)^{1/p'}\left(\sum_k^{N-1} \tau \|F^{k+1}\|_{L^p(\Omega)}\right)^{1/p}
$$
\end{remark}
\subsection{Consequences for discretized parabolic problems.}

Consider a sequence of functions $u^k\geq 0$ (nonnegativity is crucial if we limit ourselves to inequalities) satisfying 
\begin{align}
\begin{cases}
\dfrac{u^{k+1}-u^k}{\tau}-\Delta a^{k+1} u^{k+1}\leq C(1+ u^{k+1}),\\
\partial_n (a^{k+1} u^{k+1})=0,\\
u^0\geq 0, \quad u^0 \in L^\infty(\Omega).
\end{cases}
\end{align}
for nonnegative functions $a^{k+1}$ satisfying 
$$
0<a\leq a^{k+1}\leq b <+\infty,
$$
and some nonegative constant $C\geq 0$, such that $C\tau<1$, then 
\begin{lemma}\label{lem:dualite_discrete}
Let $1<p<+\infty$ be such that 
$$
\frac{b-a}{2}K_{\frac{a+b}{2},p}<1.
$$
Then, the following estimate holds true
 \begin{align}\label{eq:dualite_discrete}\left(\sum_{k=1}^N\int_\Omega \tau |u^k|^{p'}\right)^{1/p'}\leq (1-C\tau)^{-N}\left(\|u^0\|_{p'}+C N\tau|\Omega|^{1/p'}\right)\left(1+\frac{\bar D_{a,b,p}}{1-C\tau}\right)(N\tau)^{1/p'}
 \end{align}
\end{lemma}
\begin{proof}
Replacing $u^{k}$ by $v^k=(1-C\tau)^{k}u^k$ , we can replace the inequality by 
\begin{equation*}
\begin{cases}
\frac{v^{k+1}-v^k}{\tau}-\Delta \frac{a^{k+1}}{1-C\tau} u^{k+1}\leq C(1-C\tau)^{k},\\
\partial_n (a^{k+1} v^{k+1})=0,\\
v^0\geq 0, \quad v^0 \in L^\infty(\Omega).
\end{cases}
\end{equation*} 

We consider a test function $F^k\leq 0$. We introduce the solution of \eqref{eq:dual_discret}. It is straightforward that $\Psi^k\geq 0$. Therefore, multiplying one equation by $\Psi^k$ and the other by $v^{k+1}$ and summing up we have 
$$
-\sum_{k=1}^N\int_\Omega \tau v^kF^k\leq \int_\Omega v^0\Psi^0+C\sum_{k=0}^{N-1} \tau \Psi^k.
$$
By the previous results, we have then immediately 
\begin{align*}
-\sum_{k=1}^N\int_\Omega \tau v^kF^k & \leq \left(\|v^0\|_{p'}+C N\tau|\Omega|^{1/p'}\right)\max_k \|\Psi^k\|_p
\end{align*}
Combining this with remark~\ref{rem:notequal} and the fact that 
$$
\bar D_{\frac{a}{1-C\tau},\frac{b}{1-C\tau},p}=\frac{\bar D_{a,b,p}}{1-C\tau},
$$
we end up with 
\begin{align*}
-\sum_{k=1}^N\int_\Omega \tau v^kF^k&\leq \left(\|v^0\|_{p'}+C N\tau|\Omega|^{1/p'}\right)\left(1+\frac{\bar D_{a,b,p}}{1-C\tau}\right)(N\tau)^{1/p'}\left(\sum_0^{N-1} \tau \|F^{k+1}\|_{L^p(\Omega)}^p\right)^{1/p}
\end{align*}
Since $v^k\geq 0$ and the results holds for any $F^{k}\leq 0$ this leads to 
$$
\left(\sum_{k=1}^N\int_\Omega \tau |v^k|^{p'}\right)^{1/p'}\leq \left(\|v^0\|_{p'}+C N\tau|\Omega|^{1/p'}\right)\left(1+\frac{\bar D_{a,b,p}}{1-C\tau}\right)(N\tau)^{1/p'}.
$$

Which ends the proof of the lemma. 
\end{proof}
\section{Application to cross-diffusion system with bounded pressures.}\label{sec:cross_appli}

We present here the main application we have in mind concerning the discrete duality estimates. As mentioned above, one of the main difficulty is to extend estimates to the approximation procedure. 

\subsection{Small remarks on construction procedure from \cite{lepoutre2017entropic}.}

At the heart of construction procedure is the backward Euler (often called Rothe method) approximation scheme for the equation: 
\begin{align}\label{eq:Rothe}
\begin{cases}
\frac{u_i^k-u_i^{k-1}}{\tau}-\Delta p_i(U^k)u_i^k=R_i(U^k),\\
\partial_n u_i^k =0,\\
U^0\geq 0\text{ given. }
\end{cases}
\end{align}

We recall a general result on the construction procedure introduced in \cite{Lepoutre_Desvillettes_Moussa_Trescases} and extended in \cite{lepoutre2017entropic}. We adapt lemma~ from \cite{Lepoutre_Desvillettes_Moussa_Trescases}
\begin{lemma}\label{lem:approx}
Assume hypothesis \eqref{eq:p_i_reg},\eqref{eq:r_i_reg},\eqref{eq:Ainversible},\eqref{eq:masscontrol} hold true, assume $\tau$ satisifies $C_R\tau,C_\mathcal{H}\tau \leq 1/2$. Assume $U^0\geq 0, U^0\in L^\infty$ and $\int_\Omega U^0 >0$ (component by component), then there exists a sequence $(U^k)_{k\geq 1}$ solution of \eqref{eq:Rothe}. Moreover, this sequence satisifies the following properties depending on $\tau$
\begin{align*}
\forall k\geq 1,\forall p\in]1,+\infty[,\quad  p_i(U^k)u_i^k \in W^{2,p}(\Omega) ,\\
U^k\in L^\infty(\Omega;\mathbb{R}_+^I),\; \inf_\Omega \min_i u_i^k>0,\\
\end{align*}
and the following properties 
\begin{align*}
\|U^k\|_1\leq K (1-C_R\tau)^{-k},\\
\int_\Omega \mathcal{H}(U^N)+\sum_{k=1}^N\tau\int_\Omega \nabla U^{k}D^2\mathcal{H}(U^{k})DA(U^k)\nabla U^{k}\leq C(N\tau,U^0),\\
\sum_{k=1}^N\tau\int_\Omega \left(\sum_{i=1}^I p_i(U^k)u_i^k\right)\left(\sum_{i=1}^I u_i^k\right)\leq C(N\tau,U^0)
\end{align*}
Finally, if we denote $a=\min_i\inf_{R_+^I}p_i(U)>0$ and $b=\max_i\sup_{R_+^I}p_i(U)<+\infty$, for all $p>1$, such that $\frac{b-a}{2}K_{\frac{a+b}{2},p}<1$, we have 
$$
\sum_{k=1}^N \tau \left(\sum_{i=1}^I u_i^k\right)^p\leq C(N\tau,\|U^0\|_p,a,b,p).
$$
\end{lemma}
We let the reader notice that the sequence is defined for all $k>0$. We denote then the step function 
$$
U^\tau(t,x) =\sum_{k=0}^\infty U^{k+1}(x)\chi_{k\tau<t\leq (k+1)\tau}.
$$
It has been established in  \cite{Lepoutre_Desvillettes_Moussa_Trescases,lepoutre2017entropic}
that for $T>0$, we can extract a subsequence $U^{\tau_n}$ that converges almost surely to $U$. Using then the $L^2(Q_T)$ standard bounds 
\begin{align*}
U^\tau \rightarrow U \text{ in }L^r(Q_T),\quad \forall r<2\\
A(U^\tau) \rightarrow A(U) \text{ in }L^r(Q_T),\quad \forall r<2.
\end{align*}
 We can now complete this results by a $L^p$ integrability for suitable $p$.
 \begin{lemma}
The extraction $U^\tau$ converges also strongly in $L^p(Q_T)$ for any $p$ satisfying $\frac{b-a}{2}K_{\frac{a+b}{2},p}<1$. In particular, we have strong convergence for $p=2$.
\end{lemma}
\begin{proof}
Let such $p$ be given, then there exists $q>p$ such that $(b-a)K_{\frac{a+b}{2},q}<1$. Then applying lemma~\ref{lem:dualite_discrete} to $w^k=\sum u_i^k$ and $q$ where the $u_i$ are solutions to \eqref{eq:Rothe}, we obtain a uniform estimate from \eqref{eq:dualite_discrete} for any $\tau\leq 2/C_R$.
$$
\|U^\tau\|_{L^q(Q_T)}\leq e^{2C_RT}\left(\|U^0\|_{q'}+C_RT|\Omega|^{1/q'}\right)(1+2\bar D_{a,b,q})T^{1/q'}.
$$
Combining this with the almost sure convergence, we conclude that the statement holds.
\end{proof}

It has been established in solutions of \eqref{eq:Rothe} converge to a very weak solution of \eqref{eq:cross_diff_laplace} . The discrete estimate lead to estimate \eqref{eq:masscontrol}\eqref{eq:entropy_control_estimate} and \eqref{eq:dualite}. Our contribution consists here in the additional convergence properties. 
\subsection{Examples.}

 We give two last simple examples that are not covered by theorem~\ref{thm:mainLepMou} but can be covered by theorem~\ref{th:main}.
\begin{align}\label{eq:SKTboundedquadratic}
\begin{cases}
\partial_t u-\Delta \left(d_1+\frac{v}{1+v}\right)u=u(1-u-v),\\
\partial_t v-\Delta \left(d_2+\frac{u}{1+u}\right)v=v(1-v-u),\\
\partial_nu=\partial_nv =0,\quad \text{ on }\partial\Omega.
\end{cases}
\end{align}
A very close but superquadratic example is the following

\begin{align}\label{eq:SKTbounded}
\begin{cases}
\partial_t u-\Delta \left(d_1+\frac{v}{1+v}\right)u=u(1-u\log(1+u)-v),\\
\partial_t v-\Delta \left(d_2+\frac{u}{1+u}\right)v=v(1-v-u),\\
\partial_nu=\partial_nv =0,\quad \text{ on }\partial\Omega.
\end{cases}
\end{align}
In both cases, the entropy verifies 
$$
\mathcal{H}(U)=u\log\frac{2u}{1+u}+\frac{1-u}{2}+v\log\frac{2v}{1+v}+\frac{1-v}{2},\quad \nabla\mathcal{H}=\begin{pmatrix}\log\frac{2u}{1+u}+\frac{1}{1+u}-\frac{1}{2}\\ \log\frac{2v}{1+v}+\frac{1}{1+v}-\frac{1}{2}\end{pmatrix}
$$
We restrict ourselves to a $L^\infty$ initial data for sake of clarity (it is clearly not optimal). 

We let the reader check that all structural hypothesis of theorem~\ref{thm:mainLepMou} are fullfilled at the notable exception of \eqref{eq:equiint}.  To apply theorem~\ref{th:main} we  recall from lemma~\ref{lem:dualite_discrete} that there exists $p>2$ such that \eqref{eq:condKab} holds true. Therefore, there exists $p>2$ (depending on $a,b$ and $\Omega$)  such that hypothesis holds true. As a consequence, the reaction terms are in both cases uniformly equiintegrable; thanks to Vitali theorem and almost everywhere convergence they converge in $L^1(Q_T)$. Estimate \eqref{eq:gradient_control} is then an extension of its discrete equivalent in lemma~\ref{lem:approx}. Since $U^\tau$ satisifies (convention $U=U^0$ for $t\in]-\tau,0]$)
$$
\begin{cases}
\frac{U^\tau(t)-U^\tau (t-\tau)}{\tau}-\Delta A(U^\tau(t))=R(U^\tau(t)), t>0,\text{ in }\Omega\\
\partial_nU^\tau =0,\text{ on }\partial\Omega
U^\tau=U^0,\; t\in]-\tau,0].
\end{cases}
$$
Multyipling by a test function as in theorem~\ref{th:main} and integrating by parts we have 
$$
-\int_\Omega U_{in}\left(\frac{1}{\tau}\int_{-\tau}^0 \Psi(t+\tau)dt\right)=\int_0^{T-\tau}\frac{\Psi(t+\tau)-\Psi(t)}{\tau} U^\tau(t) +\int_{Q_T}\Big(A(U^\tau)\Delta\Psi + R(U^\tau)\Psi\Big)
$$
Passing to the limit we obtain \eqref{eq:veryweak}.

There is a small difference anyway between the two cases:
\begin{itemize}
\item In the first situation the reaction terms are quadratic and we need to establish some strong convergence in $L^2(Q_T)$ from the approximated solutions. This difficulty can bedealt with by employing direct $L^2$ compactness arguments see \cite{moussa2017non,pierre2010global}.
\item In the second case \eqref{eq:SKTbounded}, strong compactness in $L^2$ is not sufficient  additional integrabilty is needed and our result is necessary to ensure in particular the equiintingrability of the reaction term $u^2\log(1+u)$. 
\end{itemize}
\begin{remark}
In general, the value of constant $K_{m,p}$ (or $C_{m,p}$) is not known. Moreover, its values (for $p\not=2$) depends on the domain. The most practical (meaning independent of the domain) application of hypothesis \eqref{eq:hypRp} is the case of possibly superquadratic reaction terms but stills satisfying 
$$
|R(U)|=o(|U|^p) \quad \forall p>2.
$$
Typically the application to cubic nonlinearities in reaction terms may depend on the domain. 
\end{remark}
\section{Conclusion.}
In this manuscript we have established a time-discrete version of the  improved duality estimate from \cite{canizo2014improved}. This allows to extend a little bit known results on cross-diffusion with bounded cross-diffusion pressures. A quite important open question is the optimal possible estimate. It remains to establish in lemma~\ref{lem:dualite_discrete} can hold replacing $K_{m,p}$ by $C_{m,p}$. We think there is hope for it   up to the price of a dependence on $\tau,N$ for the correction. Typically, we have in mind that there shall be room so that the optimal constant for $N\tau$ fixed (that is $T$ is fixed) could be in the limit $\tau\rightarrow+0$ bounded by $C_{m,p}$. If such a result was established, then we would be able to extend our results to the optimal condition replacing $K_{m,p}$ by $C_{m,p}$. An important open problem is the treatment of quadratic reaction for unbounded diffusion pressures. 

\textbf{Acknowledgement} This research was supported by the project Kibord ANR-13-BS01-0004 funded by the
French Ministry of Research. 
\bibliographystyle{abbrv}
\bibliography{cms_lepoutre_b}

          \end{document}